\documentclass[12pt]{amsart}
\usepackage{latexsym,amsmath,amsthm,verbatim,ifthen,amssymb}
\usepackage[latin1]{inputenc}
\usepackage[all]{xy}
\usepackage{graphicx}
\usepackage{epsfig}
\newdir{ >}{{}*!/-7pt/\dir{>}}
\newdir{ >}{{}*!/-7pt/\dir{>}}
\newcommand{\ex}{\mathcal{E}} % \varepsilon}
\newcommand{\T}{\tau}  %\mathcal{T}\hspace{.5pt}}

\newcommand{\timess}{\bar \times}

\newcommand{\E}{\textbf{E}}

\newcommand{\N}{\mathbb{N}}

\newcommand{\SN}{\mathfrak{S}}

\newcommand{\DN}{\mathfrak{D}}

\newtheorem{theorem}{Theorem}

\newtheorem{lemma}[theorem]{Lemma}
\newtheorem{proposition}[theorem]{Proposition}

\newtheorem{definition}[theorem]{Definition}
\newtheorem{remark}[theorem]{Remark}

\newcommand{\PP}{\textbf{P}}

\begin{document}

\title[A remark on proper partitions of unity] {A
remark on proper partitions of unity.}

\author{Jose M. Garc\'{\i}a Calcines}



\address{J.M. Garc\'{\i}a Calcines \newline \indent
Departamento de Matem\'{a}tica Fundamental \newline \indent
Universidad de La Laguna \newline \indent 38271 La Laguna, Spain.}

\email{jmgarcal@ull.es}

\thanks{This work has been supported by the \emph{Ministerio de
Educaci\'on y Ciencia} grant MTM2009-12081 and FEDER}

\maketitle

\begin{abstract}
In this paper we introduce, by means of the category of exterior
spaces and using a process that generalizes the Alexandroff
compactification, an analogue notion of numerable covering of a
space in the proper and exterior setting. An application is given
for fibrewise proper homotopy equivalences.
\end{abstract}

\section*{Introduction.}
The notion of numerable covering of a space is a useful tool in
order to obtain global results from local data. Among the most
successful results in this sense for homotopy theory we can
mention the papers of A. Dold \cite{D} and T. tom Dieck \cite{tD},
in which it is shown that being a homotopy equivalence or a
fibration is a local property.
%This means that for a given map
%$f:X\rightarrow Y$ and $\{Y_i\}_{i\in I}$ a numerable cover of
%$Y$, then $f$ is a homotopy equivalence (resp. fibration) if and
%only if each $f_i:X_i\rightarrow Y_i$ is a homotopy equivalence
%(resp. fibration), where $X_i=f^{-1}(Y_i)$ and $f_i$ is the
%natural restriction.
The aim of this paper is to establish an analogous notion of
numerable covering in the category $\mathbf{P}$ of spaces and
proper maps and to give an application. In order to do so we begin
in Section 1 by giving some preliminaries definitions and results
that will be used throughout the paper. The most important tool is
the category of \emph{exterior spaces} \cite{Ext_1}. An exterior
space is nothing else but a topological space together with a
distinguished collection of open subsets verifying certain natural
conditions that capture the behavior of a neighborhood system ('at
infinity'). Then the category $\mathbf{E}$ of exterior spaces and
exterior maps appears, containing $\mathbf{P}$ as a full
subcategory. Moreover, unlike $\mathbf{P},$ the category of
exterior spaces has better categorical properties, such as having
all limits and colimits. It is for all these reasons that we have
chosen $\mathbf{E}$ as a framework for our study in $\mathbf{P}.$

In Section 2 we establish the main notion of the paper, which is
the one of proper (and exterior) numerable covering. In order to
do this, we consider a more manageable description of the category
of exterior spaces, based on the Alexandrov one-point
compactification. Then it is shown that the notion of proper
numerable (for short, p-numerable) covering is not very
restrictive. Indeed, at the end of the section we obtain the
following result:

\medskip \textbf{Proposition}. Let $X$ be a finite dimensional
locally finite CW-complex (for instance, any open differentiable
$n$-manifold or any open PL $n$-manifold). If $\{X_{\alpha
}\}_{\alpha \in A}$ is a covering of $X$ such that the family of
their interiors $\{\mbox{int}(X_{\alpha })\}_{\alpha \in A}$ also
covers $X$ and there exists $\alpha \in A$ such that the
complement $X\setminus \mbox{int}(X_{\alpha })$ is compact, then
$\{X_{\alpha }\}_{\alpha \in A}$ is a p-numerable covering of $X.$

\medskip
Finally, using such notion of p-numerable covering, we establish
in Section 3, an application. Consider $f:X\rightarrow Y$ any
proper map over a fixed space $B,$ that is, a commutative diagram
in $\mathbf{P}$:
$$\xymatrix{
{X} \ar[rr]^f \ar[dr]_p & & {Y} \ar[dl]^q \\
 & {B}  & }$$
\medskip

\textbf{Theorem.} If $\{B_{\alpha }\}_{\alpha \in A}$ is a closed
p-numerable covering of $B$ and each restriction
$$\xymatrix{
{p^{-1}(B_{\alpha })} \ar[rr]^{f_{\alpha }} \ar[dr]_{p_{\alpha }}
& & {q^{-1}(B_{\alpha })} \ar[dl]^{q_{\alpha }} \\
 & {B_{\alpha }}  & }$$ is a proper homotopy equivalence over $B_{\alpha },$
then $f:X\rightarrow Y$ is a proper homotopy equivalence over $B.$

\section{Preliminaries. Proper category and exterior spaces.}

Recall that a \textit{proper} map is a continuous map
$f:X\rightarrow Y$ such that $f^{-1}(K)$ is a compact subset of
$X,$ for every closed compact subset $K$ of $Y.$ We will denote by
$\mathbf{P}$ the category of spaces and proper maps. Proper
homotopy is defined in a natural way.

As it is well known, the category $\mathbf{P}$ has not good
categorical properties such as limits and colimits. Therefore,
many constructions can not be considered in this setting. In order
to palliate this problem in $\mathbf{P}$, exterior spaces were
introduced in \cite{Ext_1}.

\begin{definition}{\rm \cite{Ext_1} ${}$
An \textit{exterior space} $(X,\ex \subseteq \T)$ consists of a
topological space $(X,\T)$ together with a non empty family of
open sets $\ex$, called \textit{externology} which is closed by
finite intersections and, whenever $U \supseteq E$, $E \in \ex$,
$U\in \T$, then $U\in \ex$.
 We call \textit{exterior open} subset, or in short,
\textit{e-open} subset, any element $E\in \ex$. A map between
exterior spaces $f:(X,\ex \subseteq \T) \rightarrow (X',\ex'
\subseteq \T')$ is said to be \textit{exterior} if it is
continuous and $f^{-1}(E) \in \ex$, for all $E \in \ex'$.

The category of exterior spaces will be denoted by $\E$.}
\end{definition}

For a given topological space $X$ we can consider its
\textit{cocompact externology} $\ex_{cc}$ which is formed by the
family of the complements of all closed-compact subsets of $X$.
The corresponding exterior space will be denoted by $X_{cc}$. The
correspondence $X\mapsto X_{cc}$ gives rise to a full embedding
\cite[Thm. 3.2]{Ext_1}:
$$(-)_{cc}\colon \PP \hookrightarrow \E$$

Furthermore, the category $\E$ is complete and cocomplete
\cite[Thm. 3.3]{Ext_1}. For instance, the pushout in $\E$ of
 $f:X\rightarrow Y$ and $g:Y\rightarrow Z$, is the topological pushout
 $$\xymatrix{
  X \ar[d]_f \ar[r]^g
                & Z \ar[d]^{\overline{f}}  \\
  Y  \ar[r]_{\overline{g}}
                & {Y\cup _X Z}             }$$
equipped with the \textit{pushout} externology, given by those
$E\subseteq Y\cup _X Z$ for which ${\overline{g}}^{-1}(E)$ and
${\overline{f}}^{-1}(E)$ are e-open. Another example is the
\textit{product} externology in $Y\times Z$, which consists of
those open subsets which contain a product $E\times E'$ of e-open
subsets of $Y$ and $Z$ respectively. More generally, the pullback
in $\E$ of $f:Y\rightarrow X$ and $g:Z\rightarrow X$ is the
topological pullback $Y\times _X Z$ endowed with the relative
externology induced by $Y\times Z.$ (In general, the
\textit{relative} externology in $A\subseteq X$ is given by $\ex
_A:=\{E\cap A,\,\,E\in\ex _X\}$).

%These facts turn the category $\E$ into a powerful tool for the
%study of proper homotopy.

Now we introduce the following functorial construction that can be
made in this setting:
\begin{definition}{\rm \cite{Ext_1}
 Let $X$ and $Y$ be  an exterior and a topological space respectively.
 On the product space $X\times Y$ consider the following
externology: an open set $E$ is exterior if for each
 $y\in Y$ there exists an open neighborhood of  $y$, $U_y$, and an exterior open
 $E_y$ such that $E_y \times U_y \subset E$.
Then $X\timess Y$ will denote the resulting exterior space. }
\end{definition}

\begin{remark}\label{restr-prop-cil}
If $Y$ is compact, then it is not difficult to check that $E$ is
an exterior open in $X\timess Y$ if and only if it is an open set
and there exists $G\in \ex_X$ for which $G\times Y\subset E$. In
particular, if $\ex_X=\ex_{cc}^X$ and $Y$ is compact, then
$X_{cc}\timess Y=(X\times Y)_{cc}$.
\end{remark}

%\item Let   $Z$ and $Y$ be  an exterior and a
%topological space respectively. On
%$Z^Y=Hom_{\mathrm{\textbf{{\scriptsize Top}}}}(Y,Z)$ consider the
%usual compact-open topology and the externology given by those
%open sets $E$ which contain an open set $(K,G)=\{\alpha \in
%Z^Y;\alpha (K)\subseteq G\}$ where $K$ is compact and $G$ is
%e-open. If $Y$ is  compact, then $E$ is  e-open  if and only if it
%contains a subset of the form $(Y,G)$ in which $G$ is  e-open.

%  \item Let $X$ and $Z$ be exterior spaces. Then we define a
%  topological space with underlying set
%  $Z^X=Hom_{\mathrm{\textbf{{\scriptsize E}}}}(X,Z)$ equipped with the
%  topology generated by the sets of the form
%  $$(K,U)=\{\alpha \in Z^X:\alpha (K)\subset
%  U\}\hspace{5pt}\mbox{and}\hspace{5pt}
%  (L,E)=\{\alpha \in Z^X:\alpha (L)\subset E\}$$ \noindent
%  where $K\subset X$ is a compact subset, $U\subset Z$ is an open subset,
%  $L\subset X$ is an e-compact subset (i.e.,
%  $L\setminus E$ is a compact subset, for all $E$ e-open subset of $X$)
%  and $E\subset Z$ an e-open subset.
%\end{itemize}

%It is proved in \cite[Thm. 3.5]{Ext_1} that, for any pair $X,Z$ of
%exterior spaces and any locally compact space $Y$,  there exists a
%natural bijection
%$$Hom_{\mathrm{\textbf{{\scriptsize E}}}}(X\timess Y, Z)
%\cong Hom_{\mathrm{\textbf{{\scriptsize E}}}}(X,Z^Y )\:.$$
There is a \textit{cylinder functor} $- \timess I:\E \rightarrow
\E$ with natural transformations $\imath_0, \, \imath_1\colon id
\rightarrow - \timess I$ and $\rho\colon - \timess I \rightarrow
id$ obviously defined.
%Dually, the \textit{cocylinder
%functor} $(- )^I\colon \E \rightarrow \E$ also generates natural
%transformations $d_0,\, d_1\colon (- )^I \rightarrow id$ and $c:
%id \rightarrow (-)^I$ as in the classical case. By the above
%theorem, these are adjoint functors $\xymatrix{ {-\bar{\times
%}I:{\bf E}} \ar@<0.6ex>[r] & {{\bf E}:(-)^I \ar@<0.6ex>[l] }}$
This construction provides a natural way to define exterior
homotopic maps ($f\simeq _eg$) in $\E$. The notion of
\emph{exterior homotopy equivalence} comes naturally.

Observe that, since $X_{cc}\timess I=(X\times I)_{cc}$ (by remark
\ref{restr-prop-cil} above), the cylinder functor may be
restricted to the proper case, $- \timess I=- \times I:\mathbf{P}
\rightarrow \mathbf{P}.$ This functor is the one used in
\cite{Ay-D-Q} to define an $I$-category structure (in the sense of
Baues) on $\mathbf{P}_{\infty }$ using proper cofibrations.

\medskip
We can also consider the category $\mathbf{E}_B$ of exterior
spaces over a fixed object $B.$ Its objects are exterior maps
$X\rightarrow B$, called \emph{exterior spaces over $B,$} and its
morphisms, called \emph{exterior maps over $B$} are commutative
triangles in $\mathbf{E}:$
$$\xymatrix{
{X} \ar[rr]^f \ar[dr]_p & & {Y} \ar[dl]^q \\
 & {B}  & }$$
\begin{remark}
The exterior spaces and exterior maps over $B$ are also called
\emph{fibrewise exterior spaces} and \emph{fibrewise exterior
maps}, respectively.
\end{remark}

Given $f,g:X\rightarrow Y$ two exterior maps over $B$, then
\emph{$f$ is homotopic to $g$ over $B$} (or \emph{fibrewise
homotopic to $g$}), denoted $f\simeq _B g$, if there exists
$H:X\times I\rightarrow Y$ a homotopy over $B$ between $f$ and
$g.$ This means that $H$ is an exterior map such that
$qH(x,t)=p(x)$ and $H(x,0)=f(x),H(x,1)=g(x),$ for all $x\in X$ and
$t\in I.$ The homotopy over $B$ is an equivalence relation,
compatible with the composition of morphisms. The notion of\emph{
homotopy equivalence over $B$} (or \emph{fibrewise homotopy
equivalence}) is naturally defined. The fibrewise notions can be
restricted to the category $\mathbf{P}$ of spaces and proper maps.

%Unfortunately,
%as the proper cylinder has not a right adjoint (see \cite[Thm.
%1.4]{Ay-D-Q}), the exterior cocylinder functor can not be
%restricted to the proper category.

Now we give a more manageable description of the category of
exterior spaces. Such description will be crucial for our main
notion, given in the next section. We shall consider what we call
the category of $\infty $-spaces.

\begin{definition}
{\rm An \emph{$\infty $-space} is a pointed space $(X,x_{0})$ such
that $\{x_{0}\}$ is closed in $X.$ An \emph{$\infty $-map} is a
pointed map $f:(X,x_{0})\to (Y,y_{0})$ verifying
$f^{-1}(\{y_{0}\})=\{x_{0}\}.$ We write ${\bf Top}^\infty $ for
the corresponding category of $\infty $-spaces and $\infty $-maps.
}
\end{definition}

\begin{proposition}{\rm
There is an equivalence of categories ${\bf E}\simeq {\bf
Top}^\infty .$}
\end{proposition}

\begin{proof}
If $(X,\varepsilon _{X}\subset \tau _{X})$ is an exterior space
and $\infty$ is a point which does not belong to $X,$ then we
consider the pointed space $X^\infty=X\cup\{\infty\}$ with base
point $\infty$, equipped with the topology $$\tau^\infty=\tau
_{X}\cup \{E\cup\{\infty\}:E\in \varepsilon _{X}\}$$ Given $f$ any
exterior map, $f^{\infty }$ is obviously defined. Thus we obtain a
functor $(-)^{\infty }:{\bf E}\rightarrow {\bf Top}^\infty,$ which
is an equivalence of categories. Indeed, the quasi-inverse of
$(-)^{\infty }$ is the functor ${\bf Top}^{\infty }\rightarrow
{\bf E}$ defined as follows:

Let $(X,x_{0})$ be any object in ${\bf Top}^{\infty };$ then we
take the exterior space $\bar{X}=X\setminus \{x_{0}\}$ whose
topology and externology are given as
$$\tau _{\bar{X}}=\{A\setminus \{x_{0}\}:A\in \tau _{X}\},$$
$$\varepsilon _{\bar{X}}=\{A\setminus \{x_{0}\}:A\in \tau _{X}, x_{0}\in A\}.$$
Observe that $\tau _{\bar{X}}\subset \tau _{X}$ since $X\setminus
\{x_{0}\}\in \tau _{X}.$ The definition on morphisms is given by
the obvious restriction.
\end{proof}

The functor $(-)^{\infty }$ is closely related to the Alexandroff
compactification functor. Indeed, if $X$ is any topological space
and we consider $X_{cc}$ the cocompact externology, then it is
clear that $(X_{cc})^{\infty }=X^+$ is the Alexandroff
compactification of $X.$ There is a commutative diagram of
functors and categories
$$\xymatrix{
  {\mathbf{P}} \ar[dr]_{(-)^+} \ar@{^{(}->}[r]^{(-)_{cc}}
                & {\mathbf{E}} \ar[d]^{(-)^\infty}_{\simeq }  \\
                & {\mathbf{Top}^\infty}  }$$
\noindent where $(-)^+:\mathbf{P}\rightarrow
\mathbf{Top^{\infty}}$ is the functor induced by the Alexandroff
compactification construction. Consequently, the functor $(-)^+$
is a full embedding; furthermore, $(-)^+$ induces an equivalence
$$\mathbf{P}_{lcH}\simeq \mathbf{Top}^{\infty }_{cH}$$ between the
full subcategory $\mathbf{P}_{lcH}$ of $\mathbf{P}$ whose objects
are locally compact Hausdorff spaces and the full subcategory
$\mathbf{Top}^{\infty }_{cH}$ of $\mathbf{Top}^{\infty }$ whose
objects are based compact Hausdorff spaces. Here, the condition of
being Hausdorff cannot be removed. For instance, if $\mathbf{2}_S$
denotes the space given by the set $\mathbf{2}=\{0,1\}$ with the
Sierpinski topology $\tau =\{\emptyset ,\mathbf{2},\{0\}\},$ then
$\mathbf{2}_S$ does not come from the Alexandroff
compactification.

%(note that $\mathbf{2}$ is the final object in
%$\mathbf{Top^{\infty }}$)

\section{Proper and exterior numerable coverings.}

In this section we will establish the notion of proper (more
generally, exterior) numerable covering of an exterior space. In
order to obtain such notion we need the functor $(-)^{\infty
}:{\bf E}\rightarrow {\bf Top}^\infty .$ Recall that a (not
necessarily open) covering $\mathcal{U}$ of a space $X$ is said to
be \emph{numerable} if $\mathcal{U}$ admits a refinement by a
partition of unity (see \cite{D} or \cite{D-K-P} for more
details). It is well known that any open covering of a paracompact
space is numerable.

\begin{definition}\label{e-numerable}{\rm
Let $X$ be an exterior space and $\mathcal{U}=\{U_i\}_{i\in I}$ a
covering of $X.$ Then $\mathcal{U}$ is said to be an
\emph{exterior numerable covering} (for short, \emph{e-numerable
covering}) of $X$ if
$$\mathcal{U}^{\infty }=\{U_i\cup \{\infty \}\}_{i\in I}$$
\noindent is a numerable covering of the topological space
$X^{\infty }.$ In particular, when $X$ is a cocompact exterior
space (i.e., $X$ has the cocompact externology), then we say that
$\mathcal{U}$ is a \emph{proper numerable covering} (or
\emph{p-numerable covering}) of $X$.}
\end{definition}

Although at first sight the notion of e-numerable covering might
seem too restrictive we will see that it turns out to be far from
the case. Indeed, we will prove that for a wide class of exterior
spaces $X$, namely the exterior CW-complexes, we have that
$X^{\infty }$ is a paracompact space. Therefore, any open covering
of $X$ is an e-numerable covering, as long as there exists a
member of the covering which is an exterior open subset.

\begin{definition}\cite{Ext_2}
{\rm An \emph{exterior CW-complex} consists of an exterior space
$X$ together with a filtration $\emptyset =X^{-1} \subset X^0
\subset X^1 \subset \ldots \subset X_n\subset \ldots$, such that
$X$ is the colimit of this filtration and for each $n\geq 0$,
$X^n$ is obtained from $X^{n-1}$ by an exterior pushout of the
form
$$\xymatrix@C=1.6cm@R=1cm{
\sqcup_{\gamma \in \Gamma} \SN^{n-1}_\gamma \ar@{^{(}->}[d]
\ar[r]^{\sqcup_{\gamma \in \Gamma}\varphi_\gamma} & X^{n-1} \ar@{^{(}->}[d] \\
\sqcup_{\gamma \in \Gamma} \DN^{n}_\gamma \ar[r]^{\sqcup_{\gamma
\in \Gamma}\psi_\gamma} & X^n   }$$  Here $\SN^k$ denotes either
the $k$-dimensional sphere $S^k$ (with its topology as
externology) or the $k$-dimensional $\N$-sphere $\N\timess S^{k}$.
Analogously $\DN^k$ denotes either the classical disc $D^k$ or
$\N\timess D^{k},$ the $\N$-disc. We point out that the inclusion
$\SN^{k-1} \hookrightarrow \DN^{k}$ means either
$S^{k-1}\hookrightarrow D^k$ or $\N\timess S^{k-1}\hookrightarrow
\N\timess D^{k}.$}
\end{definition}

Observe that every classical CW-complex $X$ with its topology as
externology is an exterior CW-complex. Moreover, the class of
exterior CW-complexes contains many spaces in $\mathbf{P};$ for
instance, if $X$ is any finite dimensional locally finite
CW-complex, then $X_{cc}$ has the structure of an exterior
CW-complex. In order to see this fact one has just to take into
account the following result:
\begin{lemma}\label{push}{\rm \cite[Prop. 2.3]{G-G-M}
Consider the pushout in $\E$ of %exterior maps
$f:A\rightarrow X_{cc}$ and $g:A\rightarrow Y_{cc}$
$$\xymatrix{
  A \ar[d]_{f} \ar[r]^{g}
                & {Y_{cc}} \ar[d]  \\
  {X_{cc}} \ar[r]  &  {X_{cc}\cup _A Y_{cc}}           }$$
in which  $X,Y$ are Hausdorff, locally compact spaces, and $A$ is
a Hausdorff, locally compact exterior space.
 Then, the pushout externology is contained in the cocompact
externology in $X\cup _A Y.$ Furthermore, if $f$ and $g$ are
proper maps and $f$ (or $g$) is injective, then $X_{cc}\cup _A
Y_{cc}=(X\cup _A Y)_{cc}.$ \hfill $\square $}
\end{lemma}

In particular, suppose that $M$ is any open differentiable
$n$-manifold or any open PL $n$-manifold. As a differentiable
manifold $M$ admits a triangulation and therefore a structure of a
finite dimensional locally finite CW-complex. Consequently
$M_{cc}$ can be seen as an exterior CW-complex.

\begin{remark}
If $X$ is a strongly locally finite CW-complex (not necessarily
finite dimensional), then each skeleton $X^n$ equipped with its
cocompact externology is an exterior CW-complex. Furthermore, $X$
together with the colimit externology (i.e., the externology given
by those open subsets $E\subset X$ such that $E\cap X^n$ is
cocompact in $X^n,$ for all $n$) has the structure of an exterior
CW-complex. Note that, in the non finite dimensional case, the
colimit externology in $X$ need not agree with the cocompact
externology and therefore $X$ might not be considered as an space
in $\mathbf{P}$.
\end{remark}

 For an exterior CW-complex $X$, $X^{\infty }$ need not be
a classical CW-complex. For instance, if $X=\N\timess S^{k}$ we
have that $X^{\infty }=X^+$ is the Alexandorff compactification,
which is not a CW-complex since the property of being locally
contractible fails at $\infty \in X^+$ (see the figure below) \
\par \vspace{15pt}
\centerline{\epsfbox{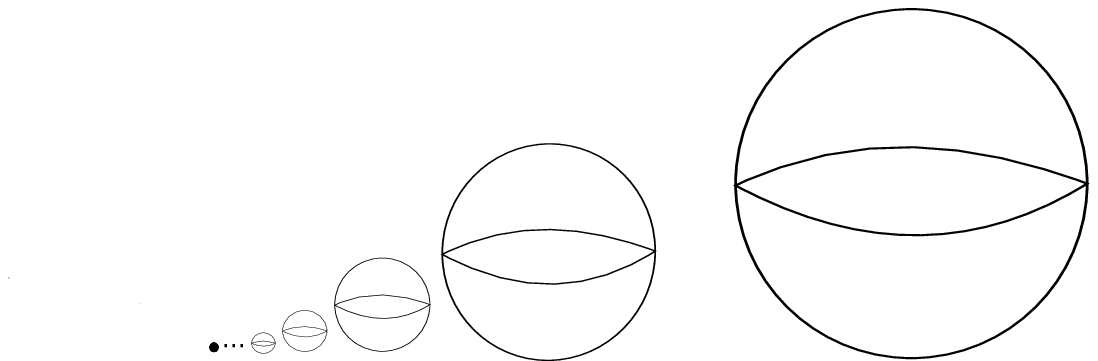}}\vspace{15pt} What is clear is
that $(\SN^k)^{\infty }=(\SN^k)^{+}$ and $(\DN^k)^{\infty
}=(\DN^k)^{+}$ are paracompact Hausdorff spaces (and therefore
normal spaces). Next we state some technical results that will
help us to reach our aim. Given $m$ any infinite cardinal number,
a topological space $X$ is said to be $m$-\emph{paracompact} if
for any open covering with cardinal $\leq m$ it admits a locally
finite refinement. Then $X$ is paracompact if and only if $X$ is
$m$-paracompact for any infinite cardinal number $m$.

A topological space $X$ is said to have the \emph{weak topology}
with respect to a closed covering $\{A_{\lambda }\}_{\lambda \in
\Lambda }$ if for any $\Lambda '\subset \Lambda $ we have that
every subset $C\subset \bigcup _{\lambda \in \Lambda '} A_{\lambda
}$ which verifies that $C\cap A_{\lambda }$ is closed for all
$\lambda \in \Lambda '$ then $C$ is closed in $X.$ Every
topological space has the weak topology with respect to any
locally finite closed covering.

\begin{lemma}{\rm \label{primero}\cite[Th.3.1]{Mo}
 If a topological space $X$ has the weak topology with respect
to a closed covering $\{A_{\alpha }\}_{\alpha \in A}$ where each
$A_{\alpha }$ is $m$-paracompact and normal, then $X$ is also
$m$-paracompact and normal. ${}$\hfill $\square $ }
\end{lemma}

\begin{lemma}{\rm \label{segundo} \cite[Th.3.4]{Mo}
Let $X$ be a topological space and $\{A_n\}_{n=0}^{\infty }$ a
countable closed covering such that if $C\subset X$ verifies that
$C\cap A_n$ is closed for all $n$ implies that $C$ is closed in
$X.$ If each $A_n$ is $m$-paracompact and normal, then $X$ is also
$m$-paracompact and normal.  ${}$\hfill $\square $}
\end{lemma}

Now recall that given any map $f:C\rightarrow Y,$ where $C$ is a
closed subespace of a space $X,$ the \emph{adjunction space}
$X\cup _fY$ is the pushout
$$\xymatrix{{C} \ar@{^{(}->}[d]  \ar[r]^f & {Y} \ar@{^{(}->}[d] \\
{X} \ar[r] &  {X\cup _fY}  }$$

\begin{lemma}{\rm \label{tercero} \cite[Cor.1]{M}
Consider $X,Y$ $m$-paracompact and normal spaces, $C$ a closed
subespace of $X$ and $f:C\rightarrow Y$ any map. Then the
adjunction space $X\cup _fY$ is $m$-paracompact and normal.}
\hfill $\square $
\end{lemma}

Using the above results one can prove the following proposition.

\begin{proposition}{\rm
If $X$ is an exterior CW-complex, then $X^{\infty }$ is a
paracompact space.}
\end{proposition}

\begin{proof}
Since $(-)^{\infty }:\mathbf{E}\rightarrow \mathbf{Top^{\infty }}$
is an equivalence of categories, in particular it preserves all
small limits and colimits. This fact implies that:
\begin{enumerate}
\item[(i)] $X^{\infty }=\mbox{colim}(X^n)^{\infty };$ that is
$X^{\infty }$ is the union of all  $(X^n)^{\infty }$ with the
hypothesis of lemma \ref{segundo}.

\item[(ii)] $(X^{n-1})^{\infty }$ and $(X^n)^{\infty }$ are related
through the topological pushout
$$\xymatrix@C=1.6cm@R=1cm{
\bigvee_{\gamma \in \Gamma} (\SN^{n-1}_\gamma)^+ \ar@{^{(}->}[d]
\ar[r] & (X^{n-1})^{\infty } \ar@{^{(}->}[d] \\
\bigvee_{\gamma \in \Gamma} (\DN^{n}_\gamma )^+ \ar[r] &
(X^n)^{\infty } }$$ That is, every $(X^n)^{\infty }$ is an
adjunction space.
\end{enumerate}
The combination of lemmas \ref{primero}, \ref{segundo} and
\ref{tercero} together with an easy induction argument permit us
to prove the result. The details are left to the reader.
\end{proof}

As corollaries we obtain the following results.

\begin{proposition}{\rm
If $\mathcal{U}$ is an open covering of an exterior CW-complex $X$
where at least one of its members is an exterior open subset, then
$\mathcal{U}$ is an e-numerable covering of $X.$ \hfill $\square
$}
\end{proposition}

\begin{proposition}{\rm
Let $X$ be a finite dimensional locally finite CW-complex (for
instance, any open differentiable $n$-manifold or any open PL
$n$-manifold). Then, its Alexandroff compactification $X^+$ is
paracompact. As a consequence, if $\{X_{\alpha }\}_{\alpha \in A}$
is a covering of $X$ such that the family of their interiors
$\{\mbox{int}(X_{\alpha })\}_{\alpha \in A}$ also covers $X$ and
there exists $\alpha \in A$ such that the complement $X\setminus
\mbox{int}(X_{\alpha })$ is compact, then $\{X_{\alpha }\}_{\alpha
\in A}$ is a p-numerable covering of $X.$ \hfill $\square $}
\end{proposition}

\section{An application: the fibrewise proper homotopy equivalences.}

In this section we will give an application. Namely, we will prove
that the fibrewise proper homotopy equivalences satisfy a local to
global type theorem. For this aim we first translate the
corresponding notions to the framework of exterior spaces, or
$\infty $-spaces. First of all, we note the following important
fact, connecting the (fibrewise) proper homotopy equivalences and
their exterior counterpart. Its proof is routine and left to the
reader.

\begin{proposition}{\rm
Let $B$ is a fixed topological space and $f:X\rightarrow Y$ a
proper map over $B,$ that is a commutative diagram in $\mathbf{P}$
$$\xymatrix{
{X} \ar[rr]^f \ar[dr] & & {Y} \ar[dl] \\
 & {B}  & }$$
Then $f$ is a proper homotopy equivalence over $B$ if and only if
$f_{cc}$ is an exterior homotopy equivalence over $B_{cc}.$ \hfill
$\square $}
\end{proposition}

We want to connect this result with the category of $\infty
$-spaces.

\subsection{Fibrewise homotopy equivalences in $\mathbf{Top}^{\infty }$}

Let $(X,x_0)$ be an $\infty $-space. Its \emph{pointed cylinder}
$I_*(X)$ is just the quotient space $I_*(X)=(X\times
I)/(\{x_0\}\times I)$ coming from the topological pushout
$$\xymatrix{
{\{x_0\}\times I} \ar[d] \ar[r] & {*} \ar[d] \\
{X\times I} \ar[r]_{\pi } & {I_*(X)} }$$ Observe that $I_{*}(X)$
is again an $\infty $-space since $\pi ^{-1}(*)=\{x_0\}\times I$
is closed in $X\times I$ and therefore $*$ is closed in $I_*(X).$
This construction gives rise to a functor
$$I_*:\mathbf{Top}^{\infty }\rightarrow \mathbf{Top}^{\infty }$$
The pointed cylinder is also equipped with natural transformations
$i_0,i_1:X\rightarrow I_*(X)$ and $p:I_*(X)\rightarrow X$ in
$\mathbf{Top}^{\infty }.$ This way a notion of homotopy comes
naturally: Given $f,g:X\rightarrow Y$ $\infty$-maps we say that
$f$ is \emph{$\infty $-homotopic} to $g$ ($f\simeq _{\infty }g$)
if there exists $F:I_*(X)\rightarrow Y$ an $\infty $-map such that
$Fi_0=f$ and $Fi_1=g.$ However we may also use the non-pointed
cylinder; it is straightforward to check that $f\simeq _{\infty
}g$ if and only if there exists a continuous map $F:X\times
I\rightarrow Y$ such that $F^{-1}(\{y_0\})=\{x_0\}\times I$ (in
particular is a classical pointed homotopy) and $F(x,0)=f(x),$
$F(x,1)=g(x),$ for all $x\in X.$ The notion of \emph{$\infty
$-homotopy equivalence} (resp. \emph{$\infty $-homotopy
equivalence over $B$}) comes naturally.

%The $\infty $-homotopy is an equivalent relation which is
%compatible with composition. Therefore we obtain the homotopy
%category $\mathbf{\pi Top}^{\infty }.$

The next result explores the connection between the cylinder
construction in ${\bf Top}^\infty$ and $\mathbf{E}.$

\begin{lemma}{\rm
Let $X$ be any exterior space. Then there exists a natural
isomorphism in ${\bf Top}^\infty $ $$(X\bar{\times }I)^{\infty
}\cong  I_*(X^{\infty })$$}
\end{lemma}

\begin{proof}
Recall that $I_*(X^{\infty })=(X^{\infty }\times
I)/(\{\infty\}\times I)$ where $\pi :X^{\infty }\times
I\rightarrow I_*(X^{\infty })$ denotes the canonical projection.
We define the $\infty $-map $\varphi :(X\bar{\times }I)^{\infty
}\rightarrow I_*(X^{\infty })$ by $\varphi (x,t)=\pi (x,t),$ for
$(x,t)\in X\times I,$ and $\varphi (\infty )=*.$ Then we have that
$\varphi $ is an isomorphism. Indeed, the map $h:X^{\infty }\times
I\rightarrow (X\bar{\times }I)^{\infty }$ \noindent given by
$h(x,t)=(x,t)$ and $h(\infty ,t)=\infty ,$ satisfies   $h(\{\infty
\}\times I)=\{\infty \}$ (moreover, $h^{-1}(\{\infty \})=\{\infty
\}\times I$) and it induces an $\infty $-map $\psi :I_*(X^{\infty
})\rightarrow (X\bar{\times }I)^{\infty }$ such that $\psi \pi
=h.$ One can straightforwardly check that $\psi \varphi =id$ and
$\varphi \psi =id.$
\end{proof}

\begin{remark}
In the proper case, the above result can be read as $$(X\times
I)^+\cong I_*(X^+)$$ That is, given any space $X,$ the Alexandroff
compactification of the cylinder $X\times I$ is, up to isomorphism
in ${\bf Top}^\infty $, the pointed cylinder of the Alexandroff
compactification of $X.$
\end{remark}

As an immediate result, given $f,g:X\rightarrow Y$ exterior maps,
we have that $f\simeq _eg$ if and only if $f^{\infty }\simeq
_{\infty }g^{\infty }.$ Moreover,

\begin{proposition}\label{traducc}{\rm
Let $B$ be a fixed exterior space and $f:X\rightarrow Y$ an
exterior map over $B.$ Then $f$ is an exterior homotopy
equivalence over $B$ if and only if $f^{\infty }$ is an $\infty
$-homotopy equivalence over $B^{\infty }.$ \hfill $\square $}
\end{proposition}

\bigskip
Now we establish our result. But first we need the following
notions and results related to the classical topological case.
Their proofs can be found in \cite{D}, \cite{D-K-P} or \cite{T}.

Recall that given $B$ a topological space, a \emph{halo} around
$A\subset B$ is a subset $V\subset B$ such that there a continuous
map $\tau :B\rightarrow [0,1]$ with $A\subset \tau ^{-1}(1)$ and
$B\setminus V\subset \tau ^{-1}(0).$ A continuous map
$p:E\rightarrow B$ is said to have the \emph{Section Extension
Property} (SEP) if for every $A\subset B$ and every section $s$
over $A$ which admits an extension as a section to a halo $V$
around $A$, there exists a section $S:B\rightarrow E$ over $B$
with $S|_A=s.$ (In particular, if $p$ has the SEP then $p$ always
has a section by taking $A=\emptyset =V.$)

Consider the category $\mathbf{Top}_B$ of topological spaces over
a fixed space $B$ and maintain the same notation and terminology
as the ones given for the exterior case. That is, we shall deal
with spaces and maps over $B$ (or fibrewise spaces and maps) and
fibrewise homotopies, also denoted as $\simeq _B.$ Then it is said
that $p:E\rightarrow B$ \emph{is dominated} by $p':E'\rightarrow
B$ if there exist fibrewise maps $f:E\rightarrow E'$ and
$g:E'\rightarrow E$ such that $gf\simeq _B id_{E}.$

By a \emph{shrinkable} space over $B$ we mean any space over $B$
which has the same fibrewise homotopy type as the identity
$id_B:B\rightarrow B.$

\begin{proposition}\label{one}\cite[Prop.2.3]{D} {\rm
Suppose that $p:E\rightarrow B$ is dominated by $p':E'\rightarrow
B.$ If $p'$ has the SEP, then so does $p.$ In particular, every
shrinkable space has the SEP. \hfill $\square $}
\end{proposition}

\begin{lemma}\label{ojoaldato}\cite[Cor.2.7]{D} {\rm
Let $p:E\rightarrow B$ a continuous map and $\{V_{\lambda
}\}_{\lambda \in \Lambda }$ a numerable covering of $B.$ If each
restriction $p_{\lambda }:p^{-1}(V_{\lambda })\rightarrow
V_{\lambda }$ is shrinkable (over $V_{\lambda }$) then $p$ is also
shrinkable. \hfill $\square $}
\end{lemma}

And the last technical previous result. Let
$$\xymatrix{
{X} \ar[rr]^f \ar[dr]_{p} & & {Y} \ar[dl]^q \\
 & {B}  & }$$
\noindent be a fibrewise map. We can consider the following
subspace of $X\times Y^I:$
$$R=\{(x,\gamma )\in X\times Y^I:p(x)=q\gamma(t),\forall t\in I,
\hspace{3pt}\mbox{and}\hspace{3pt}\gamma (1)=f(x)\}$$ \noindent
together with the map $q:R\rightarrow Y$ defined as $q(x,\gamma
)=\gamma(0).$

\begin{lemma}\label{principal}\cite[Lem.3.4]{D} {\rm
If $f:X\rightarrow Y$ is a homotopy equivalence over $B,$ then
$q:R\rightarrow Y$ is shrinkable. \hfill  $\square $}
\end{lemma}

Finally, we establish our theorem as application. Consider
$f:X\rightarrow Y$ any exterior map over a fixed exterior space
$B$
$$\xymatrix{
{X} \ar[rr]^f \ar[dr]_p & & {Y} \ar[dl]^q \\
 & {B}  & }$$
If $\{B_{\alpha }\}_{\alpha \in A}$ is covering of $B$ we will
write, for each $\alpha $
$$\xymatrix{ {X_{\alpha }}
\ar[rr]^{f_{\alpha }} \ar[dr]_{p_{\alpha }}
& & {Y_{\alpha }} \ar[dl]^{q_{\alpha }} \\
 & {B_{\alpha }}  & }$$
\noindent where $X_{\alpha }=p^{-1}(B_{\alpha }),$ $Y_{\alpha
}=q^{-1}(B_{\alpha })$ and $f_{\alpha },p_{\alpha }$ and
$q_{\alpha }$ denote the natural restrictions. \ \par
\bigskip

\begin{theorem}\label{first}
{\rm  If $\{B_{\alpha }\}_{\alpha \in A}$ is an e-numerable
covering of $B$ and each $f_{\alpha }$ is an exterior homotopy
equivalence over $B_{\alpha },$ then $f:X\rightarrow Y$ is an
exterior homotopy equivalence over $B.$}
\end{theorem}

\begin{proof}
By Proposition \ref{traducc} we can assume that we are working in
$\mathbf{Top^{\infty }},$ and by abuse of language we can use the
same notation for the corresponding objects and maps. That is, we
can think that
$$\xymatrix{
{X} \ar[rr]^f \ar[dr]_p & & {Y} \ar[dl]^q \\
 & {B}  & }$$
\noindent is a commutative diagram in $\mathbf{Top^{\infty }}$ and
$\{B_{\alpha }\}_{\alpha \in A}$ is a classical numerable covering
of $B$ such that the base point $b_0\in B$ belongs to each
$B_{\alpha } .$ Moreover, each $f_{\alpha }$ is an $\infty
$-homotopy equivalence over $B_{\alpha }.$ We shall denote by
$x_0$ and $y_0$ the respective base points of $X$ and $Y.$

Now, for every $\alpha \in A,$ let $f_{\alpha}^{-}:Y_{\alpha
}\rightarrow X_{\alpha}$ denote a fibrewise $\infty $-homotopy
inverse. As the fibrewise homotopy equivalences in
$\mathbf{Top^{\infty }}$ are, in particular, classical fibrewise
homotopy equivalences, we can use the previous results. By Lemma
\ref{principal} we have that each $q_{\alpha }:R_{\alpha
}\rightarrow X_{\alpha }$ is shrinkable, where
$$R_{\alpha }=\{(x,\gamma)\in X_{\alpha}\times
Y_{\alpha }^I:p_{\alpha}(x)=q_{\alpha }\gamma
(t),\hspace{3pt}\forall t\in I
 \hspace{3pt}\mbox{and}\hspace{3pt}\gamma (1)=f_{\alpha}(x)\}$$
Therefore, by Lemma \ref{ojoaldato} we have that
 $q:R\rightarrow Y$ is also shrinkable. In particular $q$ has the
SEP so there exists a section $S=(f',\theta ):Y\rightarrow R.$

First of all we observe that $q:R\rightarrow Y$ is an $\infty
 $-map, being $(x_0,C_{y_0})$ the base point of
 $R$ (here $C_{y_0}$ denotes the constant path).
 Indeed, if $(x,\gamma)\in R$ satisfies
 $q(x,\gamma)=\gamma(0)=y_0$ then we have
 $$p(x)=q(\gamma (0))=q(y_0)=b_0$$ \noindent so that
 $x=x_0.$ On the other hand, for any $t\in I$ we have
 $$q(\gamma (t))=p(x)=p(x_0)=b_0$$ \noindent so
 $\gamma =C_{y_0}$ is the constant path and we conclude that
 $(x,\gamma)=(x_0,C_{y_0}).$

Being $S$ a section of the $\infty $-map $q$ we also observe that,
necessarily $S:Y\rightarrow R$ must be an $\infty $-map. Indeed,
if $S(y)=(x_0,C_{y_0})$ then $y=qS(y)=q(x_0,C_{y_0})=y_0,$ so that
$S^{-1}(\{(x_0,C_{y_0})\})=\{y_0\}.$

Recall that $S$ is of the form $S(y)=(f'(y),\theta (y))$ where
$f'$ is a map $f':Y\rightarrow X$ over $B$ and $\theta $ is a map
$\theta :Y\rightarrow X^I,$ which induces, in a natural way, a
homotopy $\Theta :Y\times I\rightarrow X$ over $B$, given by
$\Theta (y,t)=\theta (y)(t).$ Since $S$ is an $\infty $-map it is
straightforward to check that $f'$ is also an $\infty $-map and
that $\Theta ^{-1}(\{x_0\})=\{y_0\}\times I.$ Moreover, $\Theta
:id_Y\simeq _B ff'$ in $\mathbf{(Top^{\infty })_B}.$

It only remains to prove that there exists $\Theta ':id_{X}\simeq
_B f'f$ in $\mathbf{(Top^{\infty })_B}.$ Applying the above
reasoning to $f'$ we have that there exist $f'':X\rightarrow Y$
over $B$ and $\Theta ':id_{X}\simeq _B f'f''$ in
$\mathbf{(Top^{\infty })_B}.$ But
$$f'f''=f'id_Yf''\simeq _Bf'ff'f''\simeq _Bf'fid_{X}=f'f$$
\noindent so  $id_{X}\simeq _B f'f.$
\end{proof}

This result has a very interesting proper counterpart. Indeed,
consider $f:X\rightarrow Y$ any proper map over a fixed space $B,$
that is, a commutative diagram of proper maps
$$\xymatrix{
{X} \ar[rr]^f \ar[dr]_p & & {Y} \ar[dl]^q \\
 & {B}  & }$$
Then we obtain as a corollary the corresponding theorem in the
proper setting:

\begin{theorem}\label{first-proper}{\rm
If $\{B_{\alpha }\}_{\alpha \in A}$ is a closed p-numerable
covering of $B$ and each restriction
$$\xymatrix{
{p^{-1}(B_{\alpha })} \ar[rr]^{f_{\alpha }} \ar[dr]_{p_{\alpha }}
& & {q^{-1}(B_{\alpha })} \ar[dl]^{q_{\alpha }} \\
 & {B_{\alpha }}  & }$$ is a proper homotopy equivalence over $B_{\alpha },$
then $f:X\rightarrow Y$ is a proper homotopy equivalence over $B.$
\hfill $\square $}
\end{theorem}

\end{document}